\documentclass[11pt,reqno]{amsart}

\usepackage{amsmath,amsthm,amssymb}
\usepackage{amssymb}
\usepackage[foot]{amsaddr}

\usepackage{graphics}
\usepackage{hyperref}
\usepackage[usenames, dvipsnames]{xcolor}

\usepackage[square,sort,comma,numbers]{natbib}

\definecolor{blue}{rgb}{0.0,0.0,0.9}
\hypersetup{colorlinks,breaklinks,
  linkcolor=blue,urlcolor=blue,
anchorcolor=blue,citecolor=blue}

\theoremstyle{plain}
\newtheorem{theorem}{Theorem}[section]
\newtheorem*{theorem*}{Theorem}

\newtheorem{proposition}[theorem]{Proposition}
\newtheorem*{proposition*}{Proposition}

\newtheorem*{corollary*}{Corollary}

\theoremstyle{definition}
\newtheorem{remark}[theorem]{Remark}
\newtheorem{definition}[theorem]{Definition}

\numberwithin{equation}{section}

\renewcommand{\Re}{\operatorname{Re}}

\title[Asymptotic Density of $k$-tuples with arbitrary gcd conditions]{On the Asymptotic Density of $k$-tuples of Positive Integers Satisfying Arbitrary GCD Conditions}

\author{Chan Ieong Kuan}
\address{School of Mathematics (Zhuhai), Sun Yat-Sen University, Zhuhai, Guangdong Province, 519082, China}
\email{kuanchi3@mail.sysu.edu.cn}

\date{\today}

\begin{document}

\maketitle

\begin{abstract}
  We consider the problem of counting $k$-tuples of positive integers satisfying any arbitrary set of gcd conditions, where every integer is not larger than $x$. We first establish the conditions to guarantee the existence of such tuples, and then obtain asymptotic formulae for the count of such tuples with the help of a multivariable Dirichlet series. Part of this work can be viewed as a generalization of T\'oth's work, where the conditions are pairwise.
\end{abstract}

\section{Introduction}

Let $\mathbb{N}$ be the set of natural numbers ($0$ not included) and $k \in \mathbb{N}$, $k \geq 2$. This paper concerns with counting $k$-tuples of positive integers satisfying any arbitrary set of gcd conditions.

The most well-known instance of this type of problem is concerning the probability of two positive integers being coprime, which is known to be $1/\zeta(2)$. Nymann \cite{NYMANN1972469} has proven the following generalization to $k$-tuples,
\[ \sum_{\substack{n_1,\dots,n_k \leq x \\ (n_1,\dots,n_k) = 1}} 1 = \frac{x^k}{\zeta(k)} + \begin{cases} O(x \log x), &\text{ if $k=2$,} \\ O(x^{k-1}), &\text{ if $k \geq 3$}. \end{cases} \]

This problem leads to various generalizations, a few of which are particularly notable. For example, the count of $k$-tuples with pairwise coprime components has been established by T\'oth \cite{Tóth01022002},
\[ \sum_{\substack{n_1,\dots,n_k \leq x \\ (n_i,n_j) = 1 \\ 1 \leq i<j \leq k}} 1 = A_k x^k + O(x^{k-1} (\log x)^{k-1}), \]
where the constant $A_k$ is
\begin{equation} \label{constAK}
  A_k = \prod_p (1-p^{-1})^{k-1} \left( 1 + \frac{k-1}{p} \right).
\end{equation}

There are a few natural variations of this setting, some examples being requiring only some pairs to be coprime (see \cite{juanR, hu2014, Toth2024}), changing pairwise conditions to $r$-wise coprimality (see \cite{Hu2013,TOTH2016105}), requiring pairwise non-coprime instead (see \cite{heyman2014,Toth2024}), among others. However, it seems that the literature so far has not allowed for arbitrary set of gcd conditions. This paper aims to fill this gap.

A good portion of this paper can be viewed as adapting ideas from T\'oth's work \cite{Toth2024} into our setting. However, T\'oth's work \cite{Toth2024} only deals with the case of pairwise relative primality. As a result, we have to come up with a way to generalize for any set of gcd conditions, which is the main innovation of this paper.

Section \ref{sectNotation} sets up the necessary notations and definitions for us to effectively write down sets of gcd conditions. Section \ref{sectResults} states the main results using said notation. Section \ref{sectAdm} deals with the problem of admissibility of a given set of gcd conditions. Section \ref{sectDensity} starts by establishing the properties of a pair of related Dirichlet series, then derives a precise expression for the series at a special point, and finally counts the number of tuples satisfying the set of gcd conditions.

\section{Notations and definitions} \label{sectNotation}
Let $k \geq 2$ be a positive integer and $S = \{1,2,\dots,k \}$. Denote the power set of $S$ as $\mathcal{P}(S)$, and $\mathcal{C}(S)$ to be the collection of subsets of $S$ in $\mathcal{P}(S)$ with less than $2$ elements, i.e.
\[ \mathcal{C}(S) = \{ \phi \} \cup \left\{ \{ i \} | i \in S \right\}. \]

In order to specify a set of gcd conditions on a tuple of positive integers in $\mathbb{N}^k$, we define the following.
\begin{definition}
  Let $\mathcal{Q}$ be a subset of $\mathcal{P}(S) \setminus \mathcal{C}(S)$. Take a function $f: \mathcal{Q} \to \mathbb{N}$. $(\mathcal{Q},f)$ is said to be \emph{a set of gcd conditions}. It places conditions on a tuple $(n_1,\dots,n_k) \in \mathbb{N}^k$ in the following way: for $T \in \mathcal{Q}$,
\[ f(T) = \operatorname{gcd} \{ n_i | i \in T \}. \]
  A set of gcd conditions $(\mathcal{Q},f)$ is said to be \emph{admissible} if there exists tuples of natural numbers that satisfy the conditions.
\end{definition}

Each element $T$ of $\mathcal{Q}$ in the above definition is a subset of $S$, and can be viewed as a subset of indices in $\{1,\dots,k\}$. As such, each element in $\mathcal{Q}$ corresponds to a gcd condition.


One can view the above definitions as some variation of generalizing the concept of a weighted graph. The following definitions also have a similar flavor.

\begin{definition}
  For a set of gcd conditions $(\mathcal{Q},f)$, an index $i \in S$ is said to be \emph{isolated} with respect to $(\mathcal{Q},f)$ if $\mathcal{Q} \subseteq \mathcal{P}(S \setminus \{i \})$.
\end{definition}

\begin{definition}
  For a set of gcd conditions $(\mathcal{Q},f)$, a subset $W$ of $S$ is said to be a \emph{cover} of $(\mathcal{Q},f)$ if for any $T \in \mathcal{Q}$, $|T \setminus W| \leq 1$.
\end{definition}

\begin{definition}
  Let $W$ be a subset of $S$. An element $x \in S$ is a \emph{neighbor of $W$ in $(\mathcal{Q},f)$} if there exists $T \in \mathcal{Q}$ such that $T \setminus W = \{ x \}$. We denote $N(W)$ to be the set of neighbors of $W$ in $(\mathcal{Q},f)$.
\end{definition}

\begin{definition}
  Let $W$ be a subset of $S$. $W$ is \emph{an independent set with respect to $(\mathcal{Q},f)$} if there does not exist $T \in \mathcal{Q}$ such that $T \subseteq W$.\\
\end{definition}

The following definition is useful for stating the counting problem.
\begin{definition}
  Let $\mathcal{G} = (\mathcal{Q},f)$ be a set of gcd conditions. Define a function $\delta_{\mathcal{G}}: \mathbb{N}^k \to \{ 0, 1\}$ as follows,
\[ \delta_{\mathcal{G}}(n_1,\dots,n_k) = \begin{cases} 1, &\text{ if $(n_1,\dots,n_k)$ satisfies the gcd conditions,} \\ 0, &\text{ otherwise.} \end{cases} \]
\end{definition}

It will be convenient for us to break the analysis into prime components, so we introduce some notation aiding us in this direction.
\begin{definition}
  Let $\mathcal{G} = (\mathcal{Q},f)$ be a set of gcd conditions. For prime $p$, define $g_p: \mathcal{Q} \to \mathbb{Z}_{\geq 0}$ and $f_p: \mathcal{Q} \to \mathbb{N}$ as follows,
  \[ g_p(T) := \operatorname{ord}_p(f(T)) \text{ and } f_p(T) := p^{g_p(T)} \text{ where } T \in \mathcal{Q}. \]
  For each $i \in S$, prime $p$, define $v_i^{(p)} := \operatorname{max} \left( \{ 0 \} \cup \{ g_p(T) | T \in \mathcal{Q}, i \in T \} \right)$. We say that a prime $p$ \emph{shows up in a set of gcd conditions $\mathcal{G} = (\mathcal{Q},f)$} if there exists $T \in \mathcal{Q}$ such that $g_p(T) > 0$. Let $\mathcal{U}_{\mathcal{G}}$ denote the set of primes showing up in the gcd conditions,
  \[ \mathcal{U}_{\mathcal{G}} := \{ p \text{ prime } | \text{ for some } T \in \mathcal{Q}, g_p(T) > 0 \}. \]
\end{definition}
We note that for $\mathcal{G} = (\mathcal{Q},f)$, for $T \in \mathcal{Q}$, $f(T) = \prod_p f_p(T)$. The quantities $v_i^{(p)}$ are crucial for checking admissibility. For any isolated index $i$ of $S$, $v_i^{(p)} = 0$.

Last but not least, we introduce a way to reduce the set of gcd conditions which helps us break the analysis into prime components.

\begin{definition}
  Let $p$ be a prime and $\mathcal{G} = (\mathcal{Q},f)$ be an admissible set of gcd conditions. Let $Z_p$ be the following set of indices,
  \[ Z_p := \{ i \in S | \text{ for some } T \in \mathcal{Q}, i \in T, \text{ for all } j \in T\setminus \{ i \}, v_j^{(p)} > g_p(T) \}. \]
  Let $S_p := S \setminus Z_p$. We define a new collection of sets $\widetilde{\mathcal{Q}_p}$ as follows,
  \[ \widetilde{\mathcal{Q}_p} = \left\{ \{ i \in T | g_p(T) = v_i^{(p)} \} | T \in \mathcal{Q} \cap \mathcal{P}(S_p) \right\}. \]
  We define $\widetilde{f_p}: \widetilde{\mathcal{Q}_p} \to \mathbb{N}$ by $\widetilde{f_p}(T_p) = 1$. We call $\widetilde{\mathcal{G}_p} := (\widetilde{\mathcal{Q}_p}, \widetilde{f_p})$ to be \emph{a reduced set of gcd conditions of $\mathcal{G}$ with respect to $p$}.
\end{definition}
We note that for a prime $p \notin \mathcal{U}_{\mathcal{G}}$, $Z_p = \phi$ and $\widetilde{\mathcal{G}_p} = (\mathcal{Q}, f_p)$.

\section{Main results} \label{sectResults}

The following theorem gives an explicit way to check for admissibility of a given set of gcd conditions.
\begin{theorem} \label{AdmThm}
  A set of gcd conditions $(\mathcal{Q},f)$ is admissible if and only if for any prime $p$, any element $T \in \mathcal{Q}$, $g_p(T) = \operatorname{min} \{ v_i^{(p)} | i \in T \}$.
\end{theorem}

The next result gives the asymptotic formula counting the number of tuples satisfying an admissible set of gcd conditions with error term.
\begin{theorem} \label{AsymThm}
  Let $\mathcal{G} = (\mathcal{Q},f)$ be an admissible set of gcd conditions. Let $W$ be a cover of $\mathcal{G}$ not including any isolated indices.

For prime $p$, we have $\widetilde{\mathcal{G}_p} := (\widetilde{\mathcal{Q}_p}, \widetilde{f_p})$, the reduced set of gcd conditions of $\mathcal{G}$ with respect to $p$. Let $I_p$ be the isolated indices of $S_p$ with respect to $\widetilde{\mathcal{G}_p}$. Define $W_p := W \setminus (Z_p \cup I_p)$, and for a subset $V_p \subseteq W_p$, $M(V_p) := N(V_p) \setminus W_p$.

We have the following asymptotic formula,
  \[ \sum_{n_1,\dots,n_k \leq x} \delta_\mathcal{G}(n_1,\dots,n_k) = A_{\mathcal{G}} x^k + O \left(x^{k-1} (\log x)^{k-1} \right), \]
where $A_{\mathcal{G}}$ is given by the following expression,
  \[ A_{\mathcal{G}} = \prod_p \left( p^{- \sum_{i \in S} v_i^{(p)}} \!\!\! \sum_{V_p \subseteq W_p} {}'  p^{-|V_p|} (1-p^{-1})^{|W_p| - |V_p| + |M(V_p)| + |Z_p|} \right). \]
The sum $\sum_{V_p \subseteq W_p}'$ here refers to summing over independent subsets $V_p$ of $W_p$.
\end{theorem}

\begin{remark}
  If all elements inside $\mathcal{Q}$ are of cardinality $2$ and $\mathcal{U}_\mathcal{G} = \phi$, we have $v_i^{(p)} = 0$ and $Z_p = \phi$ for all primes $p$. One can then verify the formula for $A_\mathcal{G}$ matches exactly the one in Theorem 9 of T\'oth's work \cite{Toth2024}.
\end{remark}

\begin{remark}
  Let $r \in \mathbb{Z}$ satisfy $2 \leq r \leq k$. Let $\mathcal{Q} = \{ T \in \mathcal{P}(S) | \, \, |T | = r \}$ and $f: \mathcal{Q} \to \mathbb{N}$ be the function $f(T) = 1$ for all $T \in \mathcal{Q}$. This is the case of requiring $r$-wise relative coprimality. In this case, we have $v_i^{(p)} = 0$ and $Z_p = \phi$ for all primes $p$. Taking our cover $W = S$, $V$ runs through all subsets of size less than $r$ and $M(V) = \phi$ here. Our formula for $A_\mathcal{G}$ reduces to
  \[ A_\mathcal{G} = \prod_p \left( \sum_{x = 0}^{r-1} {{k}\choose{x}} p^{-x} (1-p^{-1})^{k-x} \right) . \]
For $r=2$, it can be shown to match the $A_k$ constant in equation \eqref{constAK}. And with some effort, one can show that this matches the constant shown in Theorem 2.2 of \cite{TOTH2016105}.
\end{remark}

It should be noted that the power of the logarithm in Theorem \ref{AsymThm} is not optimized. In particular, by tracking through our proof, one can show that the power of the logarithm can be replaced by $\operatorname{max} \{ |N(\{i \})| i \in S \}$.

With some minor modifications to the quantities $v_i^{(p)}$, we can also accommodate gcd conditions of the form $ a | (n_{i_1},\dots, n_{i_j})$.

\section{Admissibility of a set of gcd conditions} \label{sectAdm}

In this section, we work with the assumptions of Theorem \ref{AdmThm}. From the angle of prime factorizations, it is obvious that we can consider the gcd conditions related to each prime separately. With this in mind, we tackle the proof of Theorem \ref{AdmThm}.

%

\begin{proof}[Proof of Theorem \ref{AdmThm}]
  From the definitions in Theorem \ref{AdmThm}, for any $T \in \mathcal{Q}$, any $i \in T$, $v_i^{(p)} \geq g_p(T)$. As such, we have $g_p(T) \leq \operatorname{min} \{ v_i^{(p)} | i \in T \}$.

  Assume there exists an element $T_0 \in \mathcal{Q}$ such that $g_p(T_0) < \operatorname{min} \{ v_i^{(p)} | i \in T_0 \}$. Our goal is to prove that $(\mathcal{Q},f)$ is not admissible in this case. Suppose there exists a tuple $(n_1,\dots,n_k)$ satisfying the conditions. For $i \in S$, write $n_i = \prod_p p^{a_i^{(p)}}$. Then the gcd conditions imply that for $T \in \mathcal{Q}$, we have
\begin{equation} \label{shortEqn}
  g_p(T) = \operatorname{min} \{ a_i^{(p)} | i \in T \}.
\end{equation}
This implies for an element $T \in \mathcal{Q}$, for all $i \in T$, $a_i^{(p)} \geq g_p(T)$. Fixing $i \in S$ and running over all elements $T \in \mathcal{Q}$ that contain $i$, we see that $a_i^{(p)} \geq v_i^{(p)}$. Plugging this into \eqref{shortEqn} imply
\[ g_p(T) = \operatorname{min} \{ a_i^{(p)} | i \in T \} \geq \operatorname{min} \{ v_i^{(p)} | i \in T \}. \]
This is a contradiction for $T = T_0$. Thus, $(\mathcal{Q},f)$ is not admissible.

  Now assume that for any prime $p$, any element $T \in \mathcal{Q}$, $g_p(T) = \operatorname{min} \{ v_i^{(p)} | i \in T \}$. Our goal is to prove that $(\mathcal{Q},f)$ is admissible in this case. For $i \in S$, we construct $n_i = \prod_p p^{v_i^{(p)}}$. The tuple $(n_1,\dots,n_k)$ obviously satisfies all the constraints, which shows that $(\mathcal{Q},f)$ is admissible.

\end{proof}

\section{Asysmptotic density satisfying a specified set of gcd conditions} \label{sectDensity}

In this section, we work with the assumptions of Theorem \ref{AsymThm}. Our approach here is inspired by T\'oth \cite{Toth2024}.

We first investigate properties of the following Dirichlet series,
\[ D(s_1,\dots,s_k) := \sum_{n_1,\dots,n_k \geq 1} \frac{\delta_{\mathcal{G}}(n_1,\dots,n_k)}{n_1^{s_1}\dots n_k^{s_k}}. \]
This series converges absolutely for $\Re(s_1),\dots,\Re(s_k) > 1$. Let $\mu(n)$ be the M\"obius function. We define the following convolution symbol,
\[ (\mu * \delta_\mathcal{G})(n_1,\dots,n_k) := \sum_{d_1| n_1} \dots \sum_{d_k|n_k} \left(\prod_{i=1}^k \mu \left( \frac{n_i}{d_i} \right) \right) \delta_{\mathcal{G}}(d_1,\dots,d_k), \]
and its related Dirichlet series
\[ D'(s_1,\dots,s_k) := \sum_{n_1,\dots,n_k \geq 1} \frac{(\mu * \delta_\mathcal{G})(n_1,\dots,n_k)}{n_1^{s_1}\dots n_k^{s_k}}, \]
which also converges absolutely for $\Re(s_1),\dots,\Re(s_k) > 1$. In this same region of absolute convergence, the two series are related as follows,
\begin{equation} \label{deltaGconvolution}
  D(s_1,\dots,s_k) = \left( \prod_{i=1}^k \zeta(s_i) \right) D'(s_1,\dots,s_k).
\end{equation}

Our intermediate goal will be to give a more precise set of conditions for the region of convergence of $D'(s_1,\dots,s_k)$, which will aid in our proof of Theorem \ref{AsymThm}.

\subsection{The region of convergence for $D'(s_1,\dots,s_k)$}
For prime $p$, we let $\mathcal{G}_p := (\mathcal{Q},f_p)$. We have the factorization,
\begin{equation} \label{Deqn}
 D(s_1,\dots,s_k) = \prod_p \left( \sum_{a_1,\dots,a_k \geq 0} \frac{\delta_{\mathcal{G}_p}(p^{a_1},\dots,p^{a_k})}{p^{a_1 s_1 + \dots + a_k s_k}} \right).
\end{equation}
A similar factorization is valid for $D'(s_1,\dots,s_k)$. For convenience, we denote the series inside the product as
\[ D_p(s_1,\dots,s_k) := \sum_{a_1,\dots,a_k \geq 0} \frac{\delta_{\mathcal{G}_p}(p^{a_1},\dots,p^{a_k})}{p^{a_1 s_1 + \dots + a_k s_k}}, \]
and similarly define
\[ D_p'(s_1,\dots,s_k) := \sum_{a_1,\dots,a_k \geq 0} \frac{(\mu * \delta_{\mathcal{G}_p})(p^{a_1},\dots,p^{a_k})}{p^{a_1 s_1 + \dots + a_k s_k}}. \]
We note that both $D_p(s_1,\dots,s_k)$ and $D_p'(s_1,\dots,s_k)$ converge absolutely for $\Re(s_1),\dots,\Re(s_k) > 0$. Inside this region, when investigating the convergence of the corresponding infinite products, we can ignore finitely many primes. In particular, we can ignore all the primes in $\mathcal{U}_\mathcal{G}$.

In the rest of this subsection, fix a prime $p \notin \mathcal{U}_\mathcal{G}$. Recall $W$ is a cover of $\mathcal{G}_p$. We will calculate the series $D_p(s_1,\dots,s_k)$ by grouping the terms of $D_p$ according to subsets $V = \{ i \in W | a_i \geq 1 \}$ of $W$. For each such subset, we see that $a_i = 0$ for $i \in W \setminus V$, and that $V$ has to be independent in $\mathcal{G}_p$. Furthermore, for $i \in N(V) \setminus W$, $a_i = 0$ as well. For convenience, denote $M(V) := N(V) \setminus W$.

Let $\sum_{V \subseteq W}'$ denote the sum over independent subsets $V$ of $W$. We calculate that
\begin{align*}
  D_p(s_1,\dots,s_k) &= \sum_{a_1,\dots,a_k \geq 0} \frac{\delta_{\mathcal{G}_p}(p^{a_1},\dots,p^{a_k})}{p^{a_1 s_1 + \dots + a_k s_k}} \\
  &= \sum_{V \subseteq W} {}' \sum_{\substack{a_v \geq 1, v \in V \\ a_i = 0, i \in (W \setminus V) \cup M(V) \\ a_j \geq 0, j \in S \setminus (W \cup M(V))}} \frac{1}{p^{\sum_{k \in S} a_k s_k}} \\
  &= \sum_{V \subseteq W} {}' \prod_{v \in V} p^{-s_v} \left( 1 - p^{-s_v} \right)^{-1} \prod_{j \in S \setminus (W \cup M(V))} \left( 1 - p^{-s_j} \right)^{-1}
\end{align*}
In order to facilitate the recognition of the zeta functions, we continue as follows,
\begin{align*}
  D_p(s_1,\dots,s_k) &= \prod_{i \in S} \left(1-p^{-s_i} \right)^{-1} \sum_{V \subseteq W} {}' \prod_{v \in V} p^{-s_v} \prod_{i \in (W \setminus V) \cup M(V)} (1-p^{-s_i}),
\end{align*}
which also helps us recognize that
\begin{equation} \label{DprimeEqn}
  D_p'(s_1,\dots,s_k) = \sum_{V \subseteq W} {}' \prod_{v \in V} p^{-s_v} \prod_{i \in (W \setminus V) \cup M(V)} (1-p^{-s_i}).
\end{equation}

To understand \eqref{DprimeEqn} further, we rewrite the above as:
\begin{align*}
 D_p'(s_1,\dots,s_k) &= \sum_{V \subseteq W} {}' \prod_{v \in V} \frac{1}{p^{s_v}} \prod_{i \in (W \setminus V) \cup M(V)} \left( \sum_{a_i \geq 0} \frac{\mu(p^{a_i})}{p^{a_i s_i}} \right) \\
  &=: \sum_{a_1,\dots,a_k \geq 0} \frac{ b(a_1,\dots,a_k) }{p^{\sum_{i \in S} a_i s_i}},
\end{align*}
for constants $b(a_1,\dots,a_k)$, which are only dependent on $W$ and $\mathcal{Q}$. We note the following important properties of these constants:
\begin{enumerate}
  \item $b(0,\dots,0) = 1$,
  \item $b(a_1,\dots,a_k) \neq 0$ only if $a_1,\dots,a_k \in \{ 0, 1 \}$ and $\sum_{i \in S} a_i \neq 1$.
\end{enumerate}
The first property is observed since it can only occur in the term $V = \phi$. 

For the second property, for $i \notin W$, if $a_i \geq 1$, it can only show up when $V \neq \phi$, implying a term of the form $p^{-s_v} p^{-a_is_i}$ for some $v \in V$, and hence $\sum_{i \in S} a_i \geq 2$. For $i \in W$, a term of the form $p^{-s_i}$ can only show up when $V = \phi$ or $V = \{ i \}$. However, these term show up with coefficients $-1$ and $+1$ respectively, which cancel out.

Hence, we can write
\begin{align*}
  D_p'(s_1,\dots,s_k) = 1 + \sum_{\substack{a_1,\dots,a_k \in \{0, 1\} \\ \sum_{i \in S} a_i \geq 2}} \frac{ b(a_1,\dots,a_k) }{p^{\sum_{i \in S} a_i s_i}}.
\end{align*}
Taking the product over $p \notin \mathcal{U}_\mathcal{G}$, and noting the product over $p \in \mathcal{U}_\mathcal{G}$ is a finite product, we have the following proposition.
\begin{proposition}
  Inside the region $\Re (s_1),\dots, \Re (s_k) > 0$, the function $D'(s_1,\dots,s_k)$ converges absolutely if $\Re (s_{i_1} + \dots + s_{i_j}) > 1$ for $2 \leq j \leq k$, $1 \leq i_1 < \dots < i_j \leq k$.
\end{proposition}

\subsection{The primes that show up in the gcd conditions}
In order to write down the constant in the leading term precisely, we need to investigate the local factors $D_p'(s_1,\dots,s_k)$ for primes $p \in \mathcal{U}_\mathcal{G}$. In the rest of this subsection, fix a prime $p \in \mathcal{U}_\mathcal{G}$.

We note that for $i \in Z_p$, if $a_i \neq v_i^{(p)}$, $\delta_{\mathcal{G}_p}(p^{a_1},\dots,p^{a_k}) = 0$.

Let $(p^{a_1},\dots,p^{a_k})$ satisfy $\mathcal{G}_p = (\mathcal{Q},f_p)$. For all $i \in S$, it is required that $a_i \geq v_i^{(p)}$. Thus, we write $a_i = v_i^{(p)} + b_i$. Note that for $i \in Z_p$, $b_i = 0$. We note that any gcd condition involving an index in $Z_p$ is automatically satisfied. 

Let $\widetilde{\mathcal{G}_p} := (\widetilde{\mathcal{Q}_p}, \widetilde{f_p})$ be the reduced set of gcd conditions of $\mathcal{G}$ with respect to $p$. For an element $T \in \mathcal{Q} \cap \mathcal{P}(S_p)$, we see that the original condition translates to
\[ g_p(T) = \operatorname{min} \{ v_i^{(p)} + b_i | i \in T \}. \]
Due to the admissibility of $\mathcal{G}_p$ and Theorem \ref{AdmThm}, we know that this condition reduces to
\[ \operatorname{min} \{ b_i | i \in T, g_p(T) = v_i^{(p)} \} = 0. \]
Note that the set above has at least two elements. This actually corresponds to the following element $T_p \in \widetilde{\mathcal{Q}_p}$, where
\[ T_p = \{ i \in T | g_p(T) = v_i^{(p)} \}, \text{ and } \widetilde{f_p}(T_p) = 1. \]
Given the discussion above, we can obtain the following relationship,
\[ \delta_{\mathcal{G}_p}(p^{a_1},\dots,p^{a_k}) =  \delta_{\widetilde{\mathcal{G}_p}}(p^{b_1},\dots,p^{b_k}). \]

We calculate $D_p(s_1,\dots,s_k)$ as follows,
\begin{align*}
  D_p(s_1,\dots,s_k) &= \sum_{a_1,\dots,a_k \geq 0} \frac{\delta_{\mathcal{G}_p}(p^{a_1},\dots,p^{a_k})}{p^{a_1 s_1 + \dots + a_k s_k}} \\
  &= \sum_{\substack{i \in S \\ a_i \geq v_i^{(p)}}} \frac{\delta_{\widetilde{\mathcal{G}_p}}(p^{a_1-v_1^{(p)}},\dots,p^{a_k-v_k^{(p)}})}{p^{a_1 s_1 + \dots + a_k s_k}} \\
  &= p^{- \sum_{i \in S} v_i^{(p)} s_i} \sum_{\substack{b_i \geq 0, i \in S_p \\ b_j = 0, j \in Z_p}} \frac{\delta_{\widetilde{\mathcal{G}_p}}(p^{b_1},\dots,p^{b_k})}{p^{b_1 s_1 + \dots + b_k s_k}}
\end{align*}

At this point, we have effectively reduced our problem to a similar one in the last subsection. Let $I_p$ be the isolated indices of $S_p$ with respect to $\widetilde{\mathcal{G}_p}$. Then $W_p := W \setminus (Z_p \cup I_p)$ is a cover of $\widetilde{\mathcal{G}_p}$. Calculating the series in the same way as in the last subsection, we can deduce that
\begin{align*}
  D_p(s_1,\dots,s_k) &= p^{- \sum_{i \in S} v_i^{(p)} s_i} \prod_{i \in S} (1-p^{-s_i})^{-1} \\
  &\qquad \times \sum_{V_p \subseteq W_p} {}' \prod_{v \in V_p} p^{-s_v} \prod_{i \in (W_p \setminus V_p) \cup M(V_p) \cup Z_p} (1-p^{-s_i}).
\end{align*}
It follows that
\begin{equation} \label{DprimeEqn2}
  D_p'(s_1,\dots,s_k) = p^{- \sum_{i \in S} v_i^{(p)} s_i} \sum_{V_p \subseteq W_p} {}' \prod_{v \in V_p} p^{-s_v} \prod_{i \in (W_p \setminus V_p) \cup M(V_p) \cup Z_p} (1-p^{-s_i}). 
\end{equation}
One can actually note that the whole analysis in this subsection also holds for $p \notin \mathcal{U}_\mathcal{G}$. In particular, when $p \notin \mathcal{U}_\mathcal{G}$, we have $\mathcal{G}_p = \widetilde{\mathcal{G}_p}$, $Z_p = \phi$, $v_i^{(p)} = 0$ and $I_p$ only includes isolated points that are never in $W$ to begin with. This implies that \eqref{DprimeEqn2} is valid for all primes $p$.

The value of $D'(s_1,\dots,s_k)$ at $s_1 = \dots = s_k = 1$ is particularly important for us, which can be expressed as
\begin{align}
  D'(1,\dots,1) = \prod_p \left( p^{- \sum_{i \in S} v_i^{(p)}} \!\!\! \sum_{V_p \subseteq W_p} {}'  p^{-|V_p|} (1-p^{-1})^{|W_p| - |V_p| + |M(V_p)| + |Z_p|} \right). \label{DprimeAllOnes}
\end{align}

\subsection{Proof of Theorem \ref{AsymThm}}
This subsection follows T\'oth's approach \cite{Toth2024} closely. From \eqref{deltaGconvolution}, we deduce the following equality,
\[ \delta_\mathcal{G}(n_1,\dots,n_k) = \sum_{d_1 | n_1} \dots \sum_{d_k | n_k} (\mu * \delta_\mathcal{G})(d_1,\dots,d_k). \]

Hence, we can calculate as below:
\begin{align}
  \sum_{n_1,\dots,n_k \leq x} \delta_\mathcal{G}(n_1,\dots,n_k) &= \sum_{n_1,\dots,n_k \leq x} \sum_{d_1 | n_1} \dots \sum_{d_k | n_k} (\mu * \delta_\mathcal{G})(d_1,\dots,d_k) \notag \\
  &= \sum_{d_1,\dots,d_k \leq x} (\mu * \delta_\mathcal{G})(d_1,\dots,d_k) \prod_{i=1}^k \left( \frac{x}{d_i}  + O(1) \right) \notag \\
  &=: B(x) x^k + O(E(x)), \label{FinalPlug}
\end{align}
where
\begin{align}
  B(x) :&= \sum_{d_1,\dots,d_k \leq x} \frac{(\mu * \delta_\mathcal{G})(d_1,\dots,d_k)}{d_1\dots d_k}, \label{MainTerm}\\
  E(x) :&= \sum_{\substack{u_1,\dots,u_k \in \{0,1\} \\ \sum_{i=1}^k u_i < k }} x^{\sum_{i=1}^k u_i} \sum_{d_1,\dots,d_k \leq x} \frac{|(\mu * \delta_\mathcal{G})(d_1,\dots,d_k)|}{d_1^{u_1}\dots d_k^{u_k}}. \label{ErrorTerm}
\end{align}

For the size of $E(x)$, let $u_1,\dotsm,u_k$ be fixed, and assume $u_i = 0$. Then for $j \in S \setminus \{i\}$, $(\frac{x}{d_j})^{u_j} \leq \frac{x}{d_j}$. As such, we have
\begin{align*}
  A(u_1,\dots,u_k) :=& x^{\sum_{i=1}^k u_i} \sum_{d_1,\dots,d_k \leq x} \frac{|(\mu * \delta_\mathcal{G})(d_1,\dots,d_k)|}{d_1^{u_1}\dots d_k^{u_k}} \\ 
  \leq &x^{k-1} \sum_{d_1,\dots,d_k \leq x} \frac{|(\mu * \delta_\mathcal{G})(d_1,\dots,d_k)|}{\prod_{j \in S \setminus \{i\}} d_j} \\
  \leq &x^{k-1} \prod_{p \leq x} \sum_{v_1,\dots,v_k \geq 0} \frac{|(\mu * \delta_\mathcal{G})(p^{v_1},\dots,p^{v_k})|}{p^{\sum_{j \in S \setminus \{i\}} v_j}} \\
  \ll &x^{k-1} \prod_{\substack{p \leq x \\ p \notin \mathcal{U}_\mathcal{G}}} \left( 1 + \frac{c_{i,1}}{p} + \frac{c_{i,2}}{p^2} + \dots + \frac{c_{i,k-1}}{p^{k-1}} \right),
\end{align*}
where for $1 \leq m \leq k-1$, $c_{i,m}$ are nonnegative integers dependent only on $\mathcal{G}$. In fact, one can track $c_{i,1} = \left|\{ j \in S | \{ i,j \} \in \mathcal{Q} \}\right|$. In any case, one has $c_{i,1} \leq k-1$. Hence, we have
\begin{align*}
  A(u_1,\dots,u_k) \ll& x^{k-1} \operatorname{exp}\left(\sum_{\substack{p \leq x \\ p \notin \mathcal{U}_\mathcal{G}}} \frac{k-1}{p} \right) \\
  \ll& x^{k-1} (\log x)^{k-1}.
\end{align*}
As a result, we also have $E(x) \ll x^{k-1} (\log x)^{k-1}$. All that remains is to analyze $A(x)$.

One can observe that as $x$ goes to infinity, $B(x)$ tends to $D'(1,\dots,1)$. As such, it is natural to bound the difference between $B(x)$ and $D'(1,\dots,1)$. To this end, we operate $B(x)$ as follows,
\begin{align}
  B(x) &= \sum_{d_1,\dots,d_k \geq 1} \frac{(\mu * \delta_\mathcal{G})(d_1,\dots,d_k)}{d_1\dots d_k} - \sum_{\substack{I \subseteq S \\ I \neq \phi}} B_I(x) \notag \\
  &= D'(1,\dots,1) - \sum_{\substack{I \subseteq S \\ I \neq \phi}} B_I(x), \label{Fun}
\end{align}
where
\begin{align*}
  B_I(x) &:= \sum_{\substack{d_i > x, i \in I \\ d_j \leq x, j \notin I}} \frac{(\mu * \delta_\mathcal{G})(d_1,\dots,d_k)}{d_1\dots d_k}.
\end{align*}

Following T\'oth's approach \cite{Toth2024} in analyzing $B_I(x)$, one can prove that the bound $B_I(x) \ll x^{-1} (\log x)^{k-1}$.

Using the bounds of $E(x)$ and $B_I(x)$, together with equations \eqref{Fun} and \eqref{FinalPlug}, one can conclude that
\begin{equation} \label{ConclusionEq}
  \sum_{n_1,\dots,n_k \leq x} \delta_\mathcal{G}(n_1,\dots,n_k) = D'(1,\dots,1) x^k + O \left( x^{k-1} (\log x)^{k-1} \right).
\end{equation}
Combining with equation \eqref{DprimeAllOnes}, this completes the proof of Theorem \ref{AsymThm}.

\bibliographystyle{abbrv}
\bibliography{compiled_bibliography}

\end{document}